\documentclass[11pt, psamsfonts]{amsart}

\usepackage[a4paper,margin=4cm]{geometry}
\usepackage{amsfonts,amssymb,amscd}
\usepackage{graphicx,mathrsfs} 
\usepackage{subfigure}

\newtheorem{theorem}{Theorem}[section]

\newtheorem{lemma}[theorem]{Lemma}

\newtheorem{definition}[theorem]{Definition}

\theoremstyle{remark}

\newtheorem{remark}[theorem]{\bf Remark}
\newtheorem{example}[theorem]{\bf Example}

\renewcommand{\leq}{\leqslant}
\renewcommand{\geq}{\geqslant}

\newcommand{\ptl}{\partial}

\newcommand{\rr}{\mathbb{R}}

\newcommand{\nn}{\mathbb{N}}
\newcommand{\zz}{\mathbb{Z}}


\numberwithin{equation}{section}

\begin{document}

\title[Multi-rotationally symmetric planar convex bodies]
{The maximum relative diameter for multi-rotationally symmetric planar convex bodies}

\author[A. Ca\~nete]{Antonio Ca\~nete}
\address{Departamento de Matem\'atica Aplicada I \\ Universidad de Sevilla }
\email{antonioc@us.es}
%
%

\date{\today}


\begin{abstract}
In this work we study the maximum relative diameter functional $d_M$ in the
class of multi-rotationally symmetric planar convex bodies.
A given set $C$ of this class is $k$-rotationally symmetric for $k\in\{k_1,\ldots,k_n\}\subset\nn$,
and so it is natural to consider the standard $k_i$-partition $P_{k_i}$ associated to $C$
(which is a minimizing $k_i$-partition for $d_M$ when $k_i\geq3$) and the corresponding value $d_M(P_{k_i})$.
We establish the relation among these values, characterizing the particular sets for
which all these values coincide.

\end{abstract}

\maketitle

\section{Introduction}

The class of rotationally symmetric planar convex bodies is an interesting family of sets,
which constitutes a suitable setting for studying different geometrical problems.
Recall that a planar convex body (and so, consequently compact) is \emph{rotationally symmetric}
if it is invariant under the rotation of a certain angle centered at a point (called the center of symmetry of the set),
presenting therefore nice properties from a rotational point of view.
The sets of this family appear frequently as the optimal bodies
in inequalities involving two and three classical geometric magnitudes~\cite{sa,hss},
and in other related questions (for instance, see ~\cite[Th.~4]{cianchi}).

The sets of this class have been recently considered in the following optimization problem:
given a $k$-rotationally symmetric planar convex body $C$, where $k\in\nn$, $k\geq 2$
(which indicates that $C$ is invariant for the rotation of angle $2\pi/k$),
we can consider a decomposition $P$ of $C$ into $k$ connected subsets $C_1,\ldots, C_k$.
Then, the \emph{maximum relative diameter} associated to the decomposition $P$
is given by
$$d_M(P)=\max\{D(C_i): i=1,\ldots,k\},$$
where $D(C_i)$ denotes the classical Euclidean diameter functional.
Notice that the maximum relative diameter is a functional defined for any decomposition $P$ of $C$,
providing the largest distance in the subsets given by $P$,
suggesting in some sense how large these subsets are (concerning the diameter).
In this setting, we can investigate which decompositions of $C$
give the minimal possible value for the maximum relative diameter.
In other words, we search for the subdivisions of $C$ providing the \emph{minimal} largest
distance in the corresponding subsets.

At this point, it is convenient 
to distinguish a specific type of subdivisions called \emph{$k$-partitions}:
they are decompositions given by $k$ curves meeting in an interior point of $C$,
and meeting the boundary of $C$ at different points. They arise from the fact that $C$
has an remarkable interior point, which is the center of symmetry, and so
it is natural to consider these particular decompositions,
which are originated from an arbitrary interior point.

\begin{figure}[h]
    \includegraphics[width=0.73\textwidth]{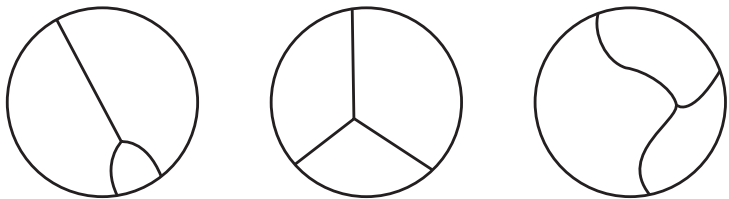}\\
  \caption{Three different $3$-partitions for the circle}
  \label{fig:circles}
\end{figure}

The above problem involving the maximum relative diameter
was studied partially for $k=2$ in \cite{mps}, proving that
a \emph{minimizing} decomposition into two equal-area connected subsets
is given by a straight line passing through the center of symmetry of the set \cite[Prop.4]{mps},
but a complete and more precise characterization is not known yet.
Later on, for $k\geq 3$, the general problem was treated in~\cite{extending}, see also \cite{trisecciones},
obtaining that the so-called \emph{standard $k$-partition} is a minimizing $k$-partition
(being also a minimizing decomposition without additional restrictions when $k\leq 6$)~\cite[Th.~4.5 and 4.6]{extending}.
We point out that the standard $k$-partition associated to a $k$-rotationally symmetric planar convex body 
divides the set into $k$ \emph{congruent} subsets by using $k$ \emph{inradius segments} (which realize the inradius of the set), see Figure~\ref{fig:st}.
In view of this result, for $k\geq 3$, the corresponding standard $k$-partition can be considered as
the optimal and more natural $k$-partition for the maximum relative diameter,
since it provides the minimal value for this geometric functional.

Apart from the above references, we mention that the maximum relative diameter functional
has been also studied in other works, obtaining interesting \emph{isodiametric inequalities}
for subdivisions into two connected subsets.
They can be found in \cite{cms2004} in the Euclidean setting,
or in \cite{cmss2010} (see also~\cite{css2013}) for compact, convex surfaces in $\rr^3$,
in some cases under the additional hypothesis of $2$-rotational symmetry.

In this context, it may happen that a given planar convex body is rotationally symmetric
under the rotation of \emph{several} angles about the center of symmetry of the set.
In this case, the invariance with respect to the different rotations gives a richer structure to the set.
For instance, a square is invariant with respect to the rotations of angles $\pi$ and $\pi/2$,
a regular hexagon is rotationally invariant for angles $\pi$, $2\pi/3$ and $\pi/3$,
and a circle is invariant under the rotation of any angle we consider. The sets satisfying
this property will be called \emph{multi-rotationally symmetric}.
We stress that, apart from the regular polygons, there are plenty of multi-rotationally symmetric planar convex bodies,
as explained in Remark~\ref{re:plenty}.

In this work we continue the study of the maximum relative diameter functional,
comparing the different values of this functional associated to the natural decompositions
of a given multi-rotationally symmetric planar convex body.
More precisely, let $C$ be a multi-rotationally symmetric planar convex body,
which is $k$-rotationally symmetric for $k\in\{k_1,\ldots,k_n\}$, with $k_i\in\nn$, $k_i\geq2$.
For each $k_i$, we can consider the standard $k_i$-partition $P_{k_i}$ associated to $C$, and the corresponding
values for the maximum relative diameter $d_M(P_{k_1}),\ldots,d_M(P_{k_n})$.
The aim of this paper is comparing these values, in order to determine which is
the minimal one, or when all of them coincide. This comparison is interesting since it
will suggest which of the optimal $k$-partitions of $C$ is the best one, in terms of the diameter.

A priori, if $k_1<\ldots<k_n$, one could think that the previous values satisfy
\begin{equation}
\label{eq:intro}
d_M(P_{k_1})>\ldots>d_M(P_{k_n}),
\end{equation}
since any subset given by $P_{k_i}$ is \emph{strictly} contained (up to rotation)
in a subset given by $P_{k_j}$ when $k_i>k_j$ (see Figure~\ref{fig:st}),
and so the monotonicity property of the diameter functional would yield that chain of \emph{strict} inequalities.

\begin{figure}[h]
 \includegraphics[width=0.73\textwidth]{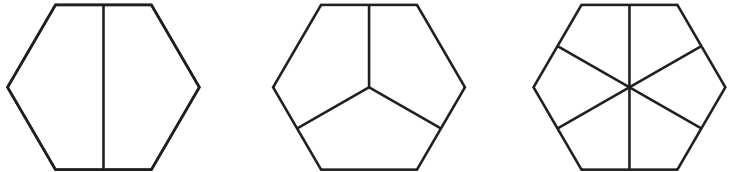}\\
  \caption{The standard $2$-partition, $3$-partition and $6$-partition for the regular hexagon}
  \label{fig:st}
\end{figure}

However, the expression~\eqref{eq:intro} is not true in general. As a consequence of Lemma~\ref{le:chain}, we shall see that
the minimal value in $\{d_M(P_{k_1}),\ldots,d_M(P_{k_n})\}$ is attained by $d_M(P_{k_n})$,
but not \emph{uniquely}. In fact, we shall check that some equality sign may appear in \eqref{eq:intro},
and we shall study when \eqref{eq:intro} is a chain of \emph{equalities}.
This situation is specially interesting since it implies that, in terms of the maximum relative diameter,
it does not matter which of the natural $k$-partitions of the set is considered,
since all of them will provide the same value for the functional. In other words,
the number $k$ of subsets given by a standard $k$-partition does not have influence for the maximum relative diameter in this case.
On the other hand, we shall also characterize when the referred minimal value is achieved uniquely.

One of the main tools used in this paper is Lemma~\ref{le:dM}, proved in \cite[Lemma~3.2]{extending},
which allows to compute the maximum relative diameter associated
to the standard $k$-partition $P_{k}$ of a $k$-rotationally symmetric planar convex body $C$
in a very easy way, when $k\geq3$. More precisely, we have that
$$d_M(P_k)=\max\{R,2\,r\sin(\pi/k)\},$$
where  $R$, $r$ are the circumradius and the inradius of $C$, respectively.
We point out that the case $k=2$ is special, since the previous formula does not hold,
and so it needs a particular treatment (recall that even a minimizing $2$-partition
is not completely characterized).

We have organized the paper as follows. In Section~\ref{sec:pre} we give the
precise definitions for posing our problem, obtaining also some technical results
that will be used later on.
In particular, we prove Lemma~\ref{le:divisors},
which describes the rotational properties of the sets of our family:
let $C$ be a multi-rotationally symmetric planar convex body,
which is $k$-rotationally symmetric for $k\in\{k_1,\ldots,k_n\}$,
being $k_n$ the largest natural number therein.
Then we have that each number $k_i$ is necessarily a divisor of $k_n$,
which will be called the maximal degree of $C$. This result implies that
for a multi-rotationally symmetric set, its associated maximal degree determines entirely
its rotationally invariant behavior.

In Section~\ref{sec:main} we prove the main results of the paper.
We shall distinguish two possibilities in our study, essentially because the case $k=2$
requires a different treatment. For a given multi-rotationally symmetric planar convex body $C$
with maximal degree $k_n\in\nn$, let $\{1,k_1,\ldots,k_n\}$ be the set of all the divisors of $k_n$,
with $k_1<\ldots<k_n$. 
In this setting, the number $k_1$ will be called the minimal degree of the set,
and will be important along our work.
We firstly assume in Subsection~\ref{subsec:k=3} that $k_1\geq3$.
In this case, we obtain in Lemma~\ref{le:chain} that
\begin{equation}
\label{eq:intro2}
d_M(P_{k_1})\geq\ldots\geq d_M(P_{k_n}),
\end{equation}
where $P_{k_i}$ represents the standard $k_i$-partition associated to $C$, $i=1,\ldots,n$.
We shall see with some examples that in the above expression \eqref{eq:intro2}
some inequalities may be strict and other ones may be equalities,
characterizing in Lemma~\ref{le:igualdad} the interesting situation where \emph{all of them}
are equalities:
\begin{quotation}
\emph{Let $C$ be a multi-rotationally symmetric planar convex body, with minimal degree $k_1\in\nn$, $k_1\geq3$.
Let $R$ be the circumradius of $C$.
Then \eqref{eq:intro2} is a chain of equalities if and only if $d_M(P_{k_1})=R$. }
\end{quotation}
This result, whose demonstration relies on Lemma~\ref{le:dM} (previously described and not valid when $k_1=2)$,
will lead us to Theorem~\ref{th:main1}, where we prove that the above condition only holds when the minimal degree is \emph{relatively} large:
\begin{quotation}
\emph{Let $C$ be a multi-rotationally symmetric planar convex body, with minimal degree $k_1\in\nn$, $k_1\geq3$.
Then, \eqref{eq:intro2} is a chain of equalities if and only if the minimal degree of $C$ is
greater than or equal to $7$.}
\end{quotation}
We shall finally obtain the following characterization in Theorem~\ref{th:main2} in terms of the maximal degree:
\begin{quotation}
\emph{Let $C$ be a multi-rotationally symmetric planar convex body, with minimal degree $k_1\in\nn$, $k_1\geq3$.
Then, \eqref{eq:intro2} is a chain of equalities if and only if the maximal degree of $C$ is
a product of prime numbers, all of them greater than or equal to $7$.}
\end{quotation}

We conclude Subsection~\ref{subsec:k=3} studying, in view of \eqref{eq:intro2},
when the minimal value in $\{d_M(P_{k_1}),\ldots,d_M(P_{k_n})\}$ is attained \emph{uniquely} by $d_M(P_{k_n})$.
In Lemma~\ref{le:minimal} we prove that this is equivalent to $d_M(P_{k_{n-1}})\neq R$,
where $R$ is the circumradius of the considered set.

In Subsection~\ref{subsec:k=2} we focus on the special case $k_1=2$.
In this situation, we have to take into account
that the associated standard $2$-partition $P_2$ is not minimizing in general,
as shown in the example from Figure~\ref{fig:dos}.
Therefore, $P_2$ cannot be considered the optimal $2$-partition for the maximum relative diameter,
and the quantity $d_M(P_2)$ will not represent the corresponding minimal value.
This means that in this case, the relations in $\{d_M(P_{k_1}),\ldots,d_M(P_{k_n})\}$
will not refer to the minimal values of our functional, as in Subsection~\ref{subsec:k=3}.
Anyway, we obtain in Lemma~\ref{le:2} that $d_M(P_{k_1})$ is the \emph{largest} value in this situation,
and consequently
$$d_M(P_{k_1})>d_M(P_{k_2})\geq\ldots\geq d_M(P_{k_n}),$$
which differs from~\eqref{eq:intro2} in the first inequality, which is always \emph{strict}.












\section{Preliminaries}
\label{sec:pre}

In this work we shall focus on rotationally symmetric planar convex bodies, assuming then the compactness of the sets.
We recall that, given $k\in\nn$, $k\geq2$, a planar convex body $C$ is said to be \emph{$k$-rotationally symmetric} if there exists a point $p\in C$
such that $C$ is invariant under the rotation of angle $2\pi/k$ about $p$. In this setting, $p$ will be referred to as the
\emph{center of symmetry} of $C$. This notion naturally suggests the following definition.

\begin{definition}
Let $C$ be a planar convex body. We will say that $C$ is multi-rotationally symmetric
if it is $k$-rotationally symmetric for more than one value of $k$.
\end{definition}

For instance, the circle is multi-rotationally symmetric since it is $k$-rotatio\-nally symmetric for any $k\in\nn$, 
and the square is also multi-rotationally symmetric since it is $2$-rotationally symmetric and $4$-rotationally symmetric.
Moreover, the regular decagon is $k$-rotationally symmetric for $k\in\{2,5,10\}$, and so it is multi-rotationally symmetric.
We stress that any set of this type possesses a rich geometric structure,
inherited by the different existing rotational symmetries leaving invariant the set.
\begin{figure}[ht]
    \includegraphics[width=0.7\textwidth]{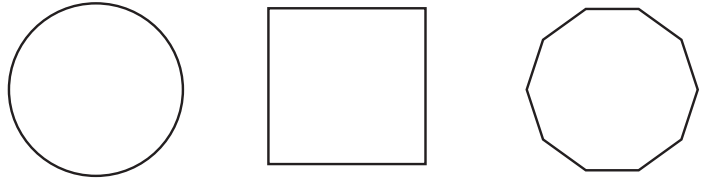}\\
  \caption{Some multi-rotationally symmetric planar convex bodies: the circle, the square and the regular decagon}
  \label{fig:examples}
\end{figure}

\begin{remark}
Any regular polygon $E_m$ of $m$ edges is $k$-rotationally symmetric for any divisor $k$ of $m$. Hence, if $m$ is not a prime number,
$E_m$ is multi-rotationally symmetric.
\end{remark}

\begin{remark}
\label{re:plenty}
Examples of multi-rotationally symmetric planar convex bodies can be constructed by the following procedure:
for any $k\in\nn$, $k\geq2$, consider the circular sector of angle $2\pi/k$, modify the curved piece of the boundary
and apply successively $k-1$ times the rotation of angle $2\pi/k$, in such a way that the resulting set $C$ is convex, see Figure~\ref{fig:modification2}.
In that case, if $k$ is not a prime number, then $C$ is multi-rotationally symmetric.
\end{remark}

\begin{figure}[h]
    \includegraphics[width=0.5\textwidth]{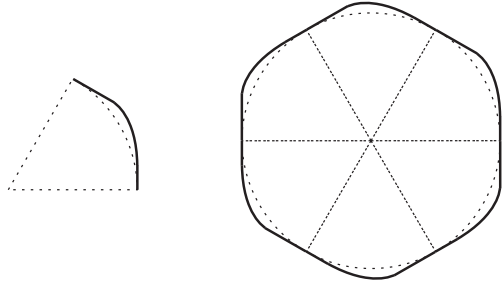}\\
  \caption{A modified circular sector of angle $2\pi/6$,
  and the resulting multi-rotationally symmetric planar convex body}
  \label{fig:modification2}
\end{figure}

Note that any multi-rotationally symmetric planar convex body is, at least,
$k$-rotationally symmetric for two different values of $k$.
If the set is not a circle, there will be a remarkable natural number associated to the set, which is defined below.
We shall see in Lemma~\ref{le:divisors} that this number determines the rotational properties of the set.

\begin{definition}
Let $C$ be a multi-rotationally symmetric planar convex body different from a circle.
The largest natural number $k$ for which $C$ is $k$-rotationally symmetric will be called the
maximal degree of $C$, and will be denoted by $k_C$.
\end{definition}

\begin{remark}
\label{re:circle1}
Any circle $\mathscr{C}$ can be seen as a \emph{degenerate} multi-rotationally symmetric set,
since it is $k$-rotationally symmetric for any $k\in\nn$, $k\geq 2$,
and so its associated maximal degree $k_\mathscr{C}$ could be set as $+\infty$. 
In fact, the circles are the only sets with this property. 
\end{remark}

\begin{lemma}
\label{le:divisors}
Let $C$ be a multi-rotationally symmetric planar convex body, with maximal degree $k_C\in\nn$.
Then $C$ is $k$-rotationally symmetric for $k\in\nn$ if and only if $k$ is a divisor of $k_C$ (greater than $1$).
\end{lemma}

\begin{proof}
It is clear that if a natural number $k>1$ is a divisor of $k_C$, then $C$ is $k$-rotationally symmetric.
Assume now that $C$ is $k$-rotationally symmetric, but $k$ is not a divisor of $k_C$.
Then $k$ and $k_C$ are coprime numbers, and so
we can find $a,b\in\zz-\{0\}$ solving the diophantine equation $k\, a+k_C\,b=1$ (see~\cite[Th.~1.7]{biggs}),
which gives
\[
\displaystyle{\varphi_{b\frac{2\pi}{k}}\circ\varphi_{a\frac{2\pi}{k_C}}=\varphi_{\frac{2\pi}{k\,k_C}}},
\]
where $\varphi_\alpha$ denotes the rotation of angle $\alpha$ about the center of symmetry of $C$.
Such an equality implies that $C$ is $(k\,k_C)$-rotationally symmetric,
which is contradictory since $k\,k_C>k_C$, and $k_C$ is the maximal degree of $C$.
\end{proof}
%
%
%
%


We now define the \emph{minimal degree} of a multi-rotationally symmetric planar convex body,
which will play an important role in Section~\ref{sec:main}.

\begin{definition}
Let $C$ be a multi-rotationally symmetric planar convex body.
The smallest natural number $k$ 
for which $C$ is $k$-rotationally symmetric will be called the
minimal degree of $C$, and will be denoted by $\chi_C$.
\end{definition}

\begin{remark}
\label{re:prime}
We point out that for any multi-rotationally symmetric planar convex body $C$,
the minimal degree $\chi_C$ is a prime number (greater than $1$), 
which will be equal to $2$ if the maximal degree of $C$ is \emph{even}, in view of Lemma~\ref{le:divisors}.
\end{remark}

The following definition describes the decompositions we shall consider for multi-rotationally
symmetric planar convex bodies.
Since this kind of sets have a special interior point (which is the center of symmetry),
it is natural, in some sense, working with a particular type of divisions called \emph{$k$-partitions}, where $k\in\nn$, see~\cite{extending}.

\begin{definition}
Let $C$ be a $k$-rotationally symmetric planar convex body, where $k\in\nn$, $k\geq2$.
A $k$-partition of $C$ is a decomposition of $C$ into $k$ connected subsets,
given by $k$ curves starting at an interior point of $C$ and meeting the boundary of $C$ at different points.
\end{definition}

\begin{remark}
We stress that, in the previous definition, the interior point of a $k$-partition
does not coincide, in general, with the center of symmetry of the set, and moreover,
the corresponding subsets need not enclose equal areas.
\end{remark}
\vspace{-2mm}

\begin{figure}[h]
   \includegraphics[width=0.77\textwidth]{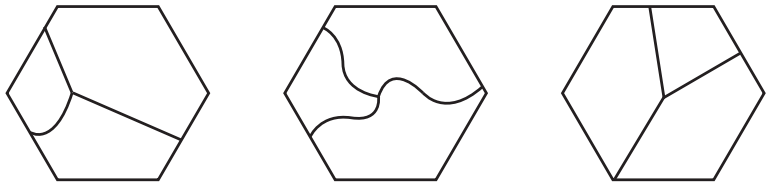}\\
  \caption{Three different $3$-partitions for the regular hexagon}
  \label{fig:partitions}
\end{figure}

We now recall the definition of the maximum relative diameter functional,
which is given by means of the classical diameter functional.

\begin{definition}
Let $C$ be a $k$-rotationally symmetric planar convex body, and let $P$ be a $k$-partition of $C$ into subsets $C_1,\ldots,C_k$.
The maximum relative diameter associated to $P$ is given by
\[
d_M(P,C)=\max\{D(C_i):i=1,\ldots,k\},
\]
where $D(C_i)=\max\{d(x,y):x,y\in C_i\}$ denotes the Euclidean diameter of $C_i$.
\end{definition}

\begin{remark}
Notice that the maximum relative diameter $d_M(P,C)$ associated to a $k$-partition $P$ of $C$
represents the largest distance in the subsets determined by $P$,
whose existence is assured due to the classical Weierstrass theorem. 
If no confusion may arise, we shall simply denote it by $d_M(P)$.
\end{remark}

Given a $k$-rotationally symmetric planar convex body $C$,
an interesting question is the study of the minimizing $k$-partitions
for the maximum relative diameter $d_M$.
That is, among all the $k$-partitions of $C$, we search for the ones providing the minimal possible value for $d_M$.
A complete characterization of a particular minimizing $k$-partition
has been recently obtained when $k\geq3$~\cite[Th.~4.5]{extending}.
We shall describe the construction of this remarkable minimizing $k$-partition,
called \emph{standard $k$-partition}, which can be considered as the optimal $k$-partition
for the maximum relative diameter functional when $k\geq3$.

%
%


Let $C$ be a $k$-rotationally symmetric planar convex body, with $k\geq3$,
and let $p$ be the center of symmetry of $C$.
Let $x_1,\ldots,x_k$ be points in $\ptl C$ at minimal distance to $p$,
which can be assumed symmetrically placed along $\ptl C$.
By considering the line segments $\overline{p\,x_i}$ (joining $p$ with each point $x_i$)
we will obtain a $k$-partition of $C$ into $k$ connected congruent subsets, see Figure~\ref{fig:standards}.
This $k$-partition is called the \emph{standard $k$-partition} associated to $C$, and will
be denoted by $P_k(C)$, or simply $P_k$.
The points $x_1,\ldots,x_n$ are called the \emph{endpoints} of $P_k$.

\begin{figure}[h]
    \includegraphics[width=0.75\textwidth]{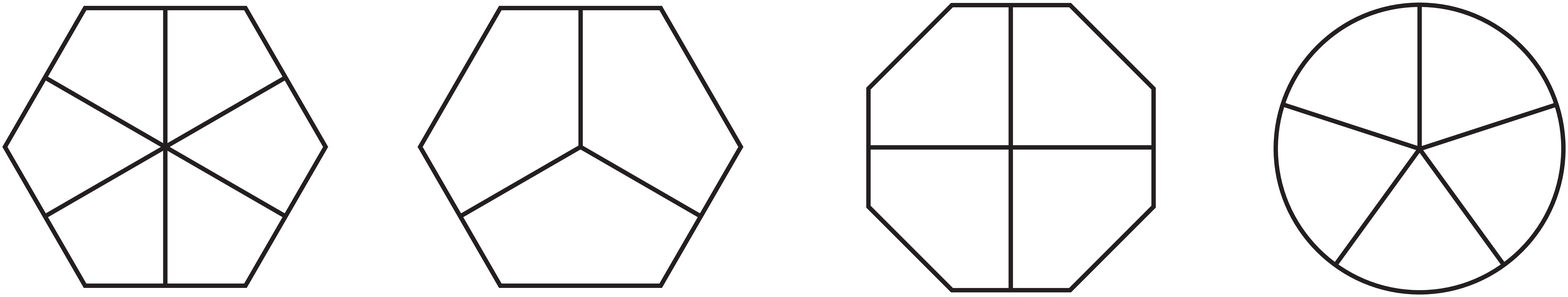}\\
  \caption{Standard $6$-partition and $3$-partition for the regular hexagon, standard $4$-partition for the regular octogon,
  and standard $5$-partition for the circle}
  \label{fig:standards}
\end{figure}

As commented previously, the standard $k$-partition plays an important role in this setting since
it minimizes the maximum relative diameter functional when $k\geq3$~\cite[Th.~4.5]{extending}.
Therefore, we can think about it as the most appropriate $k$-partition
for any $k$-rotationally symmetric planar convex body, when $k\geq3$.

The following lemma allows to compute easily the maximum relative diameter
associated to any standard $k$-partition, for $k\geq3$.

\begin{lemma}\cite[Lemma~3.2]{extending}
\label{le:dM}
Let $C$ be a $k$-rotationally symmetric planar convex body, with $k\geq 3$,
and let $P_k$ be its associated standard $k$-partition.
Then,
\[
d_M(P_k,C)=\max\{R,2\,r\sin(\pi/k)\},
\]
where $r$ and $R$ are the inradius and the circumradius of $C$, respectively.
\end{lemma}

We finish this section with the following result, which will be used later.

\begin{lemma}
\label{le:angle}
Let $C$ be a $k$-rotationally symmetric planar convex body, being $p$ its center of symmetry.
Let $x\in\ptl C$ be an endpoint of the standard $k$-partition associated to $C$.
Let $s$ be the line orthogonal to the segment $\overline{p\,x}$ passing through $x$.
Then $s$ is a supporting line of $C$.
\end{lemma}

\begin{proof}
Let $s^+$, $s^-$ be the (open) halfplanes determined by $s$, with $p\in s^-$.
Assume that there exists $q\in\ptl C$ such that $q\in s^+$.
Let $t$ be the line passing through $p$ and $x$, and $t^+$, $t^-$ the associated (open) halfplanes,
with $q\in t^+$. Consider $x'\in\ptl C\cap t^-$ close to $x$.

If $x'$ lies in $s^+\cup\,s$, then there are points in the segment $\overline{x'\,q}$ which are not contained in $C$,
contradicting the convexity of $C$. On the other hand, if $x'\in s^-$, by considering again the segment $\overline{x'\,q}$
we will get a contradiction
with the convexity of $C$ or with the definition of $x$ (which is a point of $\ptl C$ at minimal distance to $p$).
Then any point in $\ptl C$ lies necessarily in $s^-\cup\,s$ and so $s$ is a supporting line.
\end{proof}

\section{Main results}
\label{sec:main}

%

In this section we shall prove the main results of the paper.
First of all, we shall state precisely our problem.
Let $C$ be a multi-rotationally symmetric planar convex body, with associated maximal degree $k_C\in\nn$.
In view of Lemma~\ref{le:divisors}, $C$ will be $k$-rotationally symmetric for any divisor $k$ of $k_C$ (greater than $1$).
Let $\{1,k_1=\chi_C,\ldots,k_n=k_C\}$ be the set of divisors of $k_C$, with $k_1<\ldots<k_n$.

In this setting, for each divisor $k_i$ of $k_C$, with $i=1,\ldots,n$,
we shall consider the standard $k_i$-partition $P_{k_i}$ associated to $C$,
which is the \emph{optimal} $k_i$-partition for the maximum relative diameter when $k_i\geq3$ \cite[Th.~4.5]{extending},
and the corresponding value $d_M(P_{k_i})$.
In this work we investigate the relation among the values $d_M(P_{k_1}),\ldots,d_M(P_{k_n})$.

%
%

At first glance, it could seem that 
\begin{equation}
\label{eq:schain}
d_M(P_{k_1})>\ldots>d_M(P_{k_n}),
\end{equation}
with strict inequalities, since when $k_i<k_j$, each subset of $C$ given by $P_{k_j}$ is contained,
up to a proper rotation (if necessary), in a subset provided by $P_{k_i}$, see Figure~\ref{fig:a}.
\vspace{-3mm}

\begin{figure}[h]
  \includegraphics[width=0.75\textwidth]{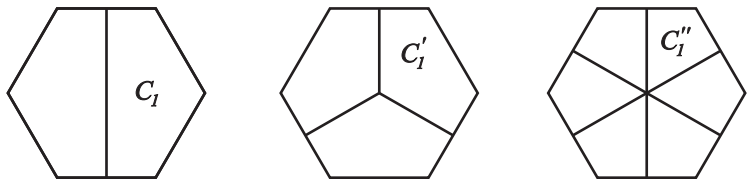}\\
  \caption{For the regular hexagon, the subsets $C_1$, $C_1'$ and $C_1''$,
  determined by the standard $2$-partition, $3$-partition and $6$-partition,
  are related by strict inclusions}
   \label{fig:a}
\end{figure}

However, we shall see that the above strict inequalities in \eqref{eq:schain} do not hold in general,
and it is even possible \emph{a chain of equalities} in \eqref{eq:schain} in some particular situations.
This case of equalities in \eqref{eq:schain} is specially interesting,
since it implies that the considered set can be divided in different natural ways
(into different numbers of connected subsets), yielding the \emph{same} value
for the maximum relative diameter.

We shall proceed by distinguishing two cases, depending on the minimal degree of our set.


\subsection{Minimal degree greater than 2}
\label{subsec:k=3}

The following result gives immediately a general chain of inequalities 
when the minimal degree of the multi-rotationally symmetric planar convex body is greater than $2$. 
The key result is Lemma~\ref{le:dM}, which only holds in this case.

\begin{lemma}
\label{le:chain}
Let $C$ be a multi-rotationally symmetric planar convex body,
with maximal degree $k_C$ and minimal degree $\chi_C\geq3$.
Let $\{k_1=\chi_C,\ldots,k_n=k_C\}$ be the set of divisors of $k_C$ greater than $1$, with $k_1<\ldots<k_n$.
Then,
\begin{equation}
\label{eq:chain}
d_M(P_{k_1})\geq\ldots\geq d_M(P_{k_n}),
\end{equation}
where $P_k$ denotes the standard $k$-partition associated to $C$.
\end{lemma}

\begin{proof}
Since $3\leq \chi_C=k_1<\ldots<k_n$, then $\sin(\pi/k_1)>\ldots>\sin(\pi/k_n)$,
and so $d_M(P_{k_1})\geq\ldots\geq d_M(P_{k_n})$ taking into account Lemma~\ref{le:dM}.
\end{proof}

\begin{remark}
\label{re:examples}
There are examples where \eqref{eq:chain} holds \emph{with strict inequalities}.
For instance, consider the regular nonagon $E_9$, which is $k$-rotationally symmetric for $k\in\{3,9\}$.
It is easy to check that $$d_M(P_3,E_9)=2\,r\sin(\pi/3)>R=d_M(P_9,E_9).$$
An analogous behavior occurs for the regular polygon $E_{15}$ of fifteen edges.
On the other hand, for some multi-rotationally symmetric planar convex bodies,
\eqref{eq:chain} may combine strict inequalities and equalities, as for the regular polygon $E_{45}$ of forty-five edges:
it is $k$-rotationally symmetric for $k\in\{3,5,15,45\}$, and straightforward computations give
$$d_M(P_3,E_{45})>d_M(P_5,E_{45})>d_M(P_{15},E_{45})=d_M(P_{45},E_{45}).$$
\end{remark}

In view of Remark~\ref{re:examples}, the following question arises:
which multi-rotationa\-lly symmetric planar convex bodies provide \emph{a chain of equalities} in \eqref{eq:chain}?
Note that for such a set, all its corresponding standard $k$-partitions will give the
same value for the maximum relative diameter. 
In other words, among all the natural ways of dividing the set, the minimum value for
the maximum relative diameter will be independent from the number of subsets of the division.
The following result will allow to characterize this kind of sets.

\begin{lemma}
\label{le:igualdad}
Let $C$ be a multi-rotationally symmetric planar convex body, with 
minimal degree $\chi_C\geq3$,  
and let $R$ be its circumradius.
Then, we have a chain of equalities in~\eqref{eq:chain} if and only if $d_M(P_{\chi_C})=R$.
\end{lemma}

\begin{proof}
Let $k_C$ be the maximal degree of $C$,
and let $\{k_1=\chi_C,\ldots,k_n=k_C\}$ be the set of divisors of $k_C$ greater than $1$,
with $k_1<\ldots<k_n$.
Assume firstly that $d_M(P_{\chi_C})=R$. Then, \eqref{eq:chain} turns
$$R=d_M(P_{k_1})\geq d_M(P_{k_2})\geq\ldots\geq d_M(P_{k_n})=\max\{R,2\,r\sin(\pi/k_n)\}\geq R,$$
by using Lemma~\ref{le:dM}, where $r$ is the inradius of $C$.
Thus $d_M(P_{k_i})$ equals $R$, for any $i\in\{1,\ldots,n\}$.

Assume now that we have a chain of equalities in \eqref{eq:chain}. 
Due to Lemma~\ref{le:dM}, we have that $d_M(P_{k_{i}})=\max\{R,2\,r\sin(\pi/k_i)\}$, $i=1,\ldots,n$.
Taking into account that $\sin(\pi/k_1)>\sin(\pi/k_2)>\ldots>\sin(\pi/k_n)$,
the only admissible possibility in this case is $d_M(P_{k_i})=R$, for any $i\in\{1,\ldots,n\}$.
\end{proof}

The previous Lemma~\ref{le:igualdad} allows to find out the values of the minimal degree
for which all equalities hold in \eqref{eq:chain}.
The following result shows that the above condition is satisfied
when the minimal degree is greater than or equal to $6$.

\begin{lemma}
\label{le:6}
Let $C$ be a multi-rotationally symmetric planar convex body.
If the minimal degree $\chi_C$ of $C$ is greater than or equal to $6$, then $d_M(P_{\chi_C})=R$.
\end{lemma}

\begin{proof}
Let $R$ and $r$ be the circumradius and the inradius of $C$, respectively.
As $\chi_C\geq 6$, it is clear that $\sin(\pi/\chi_C)\leq\sin(\pi/6)=1/2$,
and so $2\,r\sin(\pi/\chi_C)\leq r\leq R$.
Then $d_M(P_{\chi_C})=\max\{R,2\,r\sin(\pi/\chi_C)\}=R$, by using Lemma~\ref{le:dM},
which yields the statement.
\end{proof}

We point out that, in view of Remark~\ref{re:prime}, the previous Lemma~\ref{le:6} applies
in fact for $\chi_C\geq7$, since $6$ is not a prime number.
On the other hand, when $\chi_C\leq 5$, we only have to analyze the cases $\chi_C=3$ and $\chi_C=5$,
since $4$ is neither prime.
The following results show that we do not have a chain of equalities in \eqref{eq:chain} in any of these two cases.

\begin{lemma}
\label{le:3}
Let $C$ be a multi-rotationally symmetric planar convex body, and let $R$ be its circumradius.
If the minimal degree $\chi_C$ of $C$ is equal to $3$, then $d_M(P_{\chi_C})\neq R$.
\end{lemma}

\begin{proof}
Recall that, by Lemma~\ref{le:dM}, $d_M(P_{\chi_C})=\max\{R,2\,r\sin(\pi/\chi_C)\}$, where $r$ is the inradius of $C$.
Suppose that $d_M(P_{\chi_C})=R$. Then $R\geq 2\,r\sin(\pi/\chi_C)$ and so $R/r\geq\sqrt{3}$.

Fix two consecutive endpoints $v_1,\ v_2$ of $P_{k_C}$, where $k_C$ is the maximal degree of $C$,
and let $x_R\in\ptl C$ be a point with $d(p,x_R)=R$,
which can be assume to lie in the piece of $\ptl C$ delimited by $v_1$ and $v_2$.
Let $\alpha_1$ be the angle determined by the segments $\overline{p\,v_1}$ and $\overline{p\,x_R}$,
and $\alpha_2$ the angle determined by $\overline{p\,x_R}$ and $\overline{p\,v_2}$.
Since $\alpha_1+\alpha_2=2\pi/k_C$ due to the existing rotational symmetry, we can assume without loss of generality that $\alpha_1\leq\pi/k_C$.

Let us now consider the triangle with vertices $p$, $v_1$, $x_R$,
with associated angles $\alpha_1$, $\beta$, $\gamma$,
which will add up to $\pi$ radians.

\begin{figure}[h]
    \includegraphics[width=0.34\textwidth]{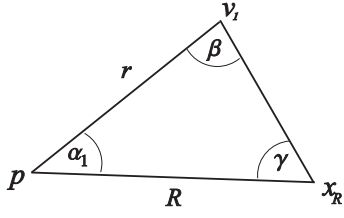}\\
  \caption{Triangle with vertices $p$, $v_1$, $x_R$}
  \label{fig:angles}
\end{figure}

By using the sine's theorem we get
$$\sqrt{3}\leq \frac{R}{r}=\frac{\sin(\beta)}{\sin(\gamma)}\leq\frac{1}{\sin(\gamma)}, $$
which gives $\sin(\gamma)\leq1/\sqrt{3}$, and so $\gamma\leq \arcsin(1/\sqrt{3})$.
Moreover, from Lemma~\ref{le:angle} we have that $\beta\leq\pi/2$. Then
\[
\pi=\alpha_1+\beta+\gamma\leq\pi/k_C+\pi/2+\arcsin(1/\sqrt{3}),
\]
which yields
\[
k_C\leq\frac{\pi}{\pi/2-\arcsin(1/\sqrt{3})}<4,
\]
which is not possible, since there are no multiples of $\chi_C=3$ satisfying that condition.
Then $d_M(P_{\chi_C})\neq R$, which finishes the proof.
\end{proof}

\begin{lemma}
\label{le:5}
Let $C$ be a multi-rotationally symmetric planar convex body, and let $R$ be its circumradius.
If the minimal degree $\chi_C$ of $C$ is equal to $5$, then $d_M(P_{\chi_C})\neq R$.
\end{lemma}

\begin{proof}
The proof is analogous to the one from Lemma~\ref{le:3}, taking into account that now,
by assuming  $d_M(P_{\chi_C})=R$, we shall get
$$R/r\geq2\sin(\pi/5)=\sqrt{\frac{5-\sqrt{5}}{2}},$$
which finally yields $k_C<6$, a contradiction.
\end{proof}


The following Theorem~\ref{th:main1} summarizes the previous lemmata,
characterizing when \eqref{eq:chain} is a chain of equalities in the case that
the minimal degree is greater than or equal to $3$.

\begin{theorem}
\label{th:main1}
Let $C$ be a multi-rotationally symmetric planar convex body with minimal degree $\chi_C\geq3$.
Then \eqref{eq:chain} is a chain of equalities if and only if $\chi_C\geq7$.
\end{theorem}

The characterization from Theorem~\ref{th:main1} is stated in terms of the
minimal degree. Regarding the maximal degree, we have Theorem~\ref{th:main2},
obtained immediately by using the following algebraic result.
%



\begin{lemma}
\label{le:algebra}
Let $k_C\in\nn$, and let $\{1,k_1,\ldots,k_C\}$ be the set of ordered divisors of $k_C$,
with $k_1\geq 7$ and $k_1\neq k_C$. Then, $k_C$ is a product of prime numbers, all of them greater than or equal to $7$.
\end{lemma}

\begin{proof}
As $k_1$ is the smallest divisor of $k_C$ (greater than $1$), it is clear that $k_1$ has to be a prime number,
and therefore, all the factors in the decomposition of $k_C$ into prime numbers will be greater than
or equal to $7$.%
\end{proof}

\begin{theorem}
\label{th:main2}
Let $C$ be a multi-rotationally symmetric planar convex body,
with maximal degree $k_C$ and minimal degree $\chi_C\geq3$.
Then, \eqref{eq:chain} is a chain of equalities if and only if
$k_C$ is a product of prime numbers, all of them greater than or equal to $7$.
\end{theorem}

\begin{example}
For instance, if a multi-rotationally symmetric planar convex body has minimal degree equal to $7$,
then we will have a chain of equalities in \eqref{eq:chain}, due to Theorem~\ref{th:main1},
and the admissible values for its maximal degree will be $7\cdot{7}=49$, $7\cdot11=77$, $7\cdot13=91$, and so on,
in view of Theorem~\ref{th:main2}. Notice that, in particular, Theorem~\ref{th:main2} implies that \eqref{eq:chain}
will be a chain of equalities only for \emph{relatively large} values of the associated maximal degree.
\end{example}


We finish this Subsection~\ref{subsec:k=3} by studying the following related question.
Let $C$ be a multi-rotationally symmetric planar convex body, with maximal degree $k_C$ and minimal degree $\chi_C\geq3$.
Denote by $\{k_1=\chi_C,\ldots,k_n=k_C\}$ the set of divisors of $k_C$ greater than $1$.
In this setting, we can search for the minimal value in $\{d_M(P_{k_1}),\ldots,d_M(P_{k_n})\}$,
where $P_{k_i}$ is the standard $k_i$-partition associated to $C$,
which represents the \emph{global minimal value} for the maximum relative diameter,
taking into account \cite[Th.~4.5]{extending}.
Additionally, this will give, in some sense, a comparison among the \emph{optimal} $k_i$-partitions of $C$
in terms of our functional.
From Lemma~\ref{le:chain}, it clearly follows that
$$\min\{d_M(P_{k_1}),\ldots,d_M(P_{k_n})\}=d_M(P_{k_n}),$$
but we can consider a further question:
is there any other standard $k_i$-partition, apart from $P_{k_n}$, attaining also that minimum value?
The following lemma answers this question.

\begin{lemma}
\label{le:minimal}
Let $C$ be a multi-rotationally symmetric planar convex body,
with maximal degree $k_C$ and minimal degree $\chi_C\geq3$.
Let $\{k_1=\chi_C,\ldots,k_n=k_C\}$ be the set of divisors of $k_C$ greater than $1$,
with $k_1<\ldots<k_n$, and denote by $P_{k_i}$ the standard $k_i$-partition of $C$, $i=1,\ldots,n$.
Then, the minimal value in
$$\{d_M(P_{k_1}),\ldots,d_M(P_{k_n})\}$$
is uniquely attained by $d_M(P_{k_n})$ if and only if $d_M(P_{k_{n-1}})\neq R$.
\end{lemma}

\begin{proof}
Taking into account Lemma~\ref{le:chain},
the considered minimal value is unique\-ly attained by $d_M(P_{k_n})$
if and only if $d_M(P_{k_{n-1}})>d_M(P_{k_{n}})$.
Recall that $d_M(P_k)=\max\{R,2\,r\sin(\pi/k)\}$,
for any $k\in\{k_1,\ldots,k_n\}$, due to Lemma~\ref{le:dM}.

Assume that $d_M(P_{k_{n-1}})\neq R$. Then $d_M(P_{k_{n-1}})=2\,r\sin(\pi/k_{n-1})>R$.
Since $2\,r\sin(\pi/k_{n-1})>2\,r\sin(\pi/k_{n})$, both inequalities yield
$$d_M(P_{k_{n-1}})=2\,r\sin(\pi/k_{n-1})>\max\{R,2\,r\sin(\pi/k_n)\}=d_M(P_{k_n}),$$
as desired.

Assume now that $d_M(P_{k_{n-1}})=R$. Then $R\geq 2\,r\sin(\pi/k_{n-1})>2\,r\sin(\pi/k_{n})$,
and so $d_M(P_{k_n})=\max\{R,2\,r\sin(\pi/k_n)\}=R$, which implies that the referred uniqueness does not hold.
\end{proof}


%
%
%
%
%
%
%

\subsection{Minimal degree equal to $2$}
\label{subsec:k=2}

If we consider a multi-rotationally symmetric planar convex body $C$ with minimal degree equal to $2$,
the situation is different from the one corresponding to Subsection \ref{subsec:k=3}.
The reason is that $C$ is, in particular, $2$-rotationally symmetric,
and so the \emph{optimal $2$-partition} for the maximum relative diameter 
is not completely characterized in this case.
Recall that it is proved in \cite{mps} that a minimizing $2$-partition (into two subsets of equal areas)
will consist of a line segment passing through the center of symmetry of the set,
but a more precise description is not known.
In fact, the corresponding standard $2$-partition is not minimizing in general, see Figure~\ref{fig:dos}.



\begin{figure}[h]
    \includegraphics[width=0.6\textwidth]{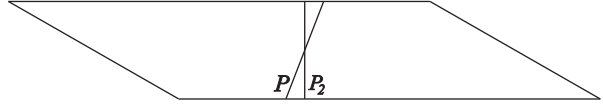}\\
  \caption{An example where the standard $2$-partition $P_2$ is not minimizing for $d_M$, 
  since $d_M(P_2)>d_M(P)$, where $P$ is obtained by a slight rotation of $P_2$}
  \label{fig:dos}
\end{figure}

Therefore, in this case we cannot make the discussion on the \emph{best values}
for the maximum relative diameter functional, as in Subsection~\ref{subsec:k=3}.
At least, we can study our problem partially,
investigating the relation among the values 
$$\{d_M(P_{k_1}),\ldots, d_M(P_{k_n}\},$$
where $\{k_1=2,\ldots,k_n\}$ is the set of divisors (greater than $1$) of the associated maximal degree of $C$,  
taking into account that $d_M(P_{k_1})$ is not optimal in general (although illustrative enough).
The existing relation is given by the following result.



\begin{lemma}
\label{le:2}
Let $C$ be a multi-rotationally symmetric planar convex body 
(different from a circle), with maximal degree $k_C$ and minimal degree $\chi_C=2$.
Let $\{k_1=\chi_C,k_2,\ldots,k_n=k_C\}$ be the set of divisors of $k_C$ (greater than $1$),
with $k_1<k_2<\ldots<k_n$.
Then, $$d_M(P_2)>d_M(P_{k_2})\geq\ldots\geq d_M(P_{k_n}),$$
where $P_{k}$ is the standard $k$-partition associated to $C$.
\end{lemma}

\begin{proof}
It suffices to prove that $d_M(P_2)>d_M(P_{k_2})$,
since the other inequalities follow as in the proof of Lemma~\ref{le:chain} (notice that $k_2>2$).
Let $v_1$, $v_2$ be the endpoints of the standard $2$-partition $P_2$.
On the other hand, by Lemma~\ref{le:dM}, we have that
$d_M(P_{k_2})=\max\{R,2\,r\sin(\pi/k_{2})\}$, where $R$ and $r$ are the
circumradius and the inradius of $C$, respectively. We shall distinguish two cases:

If $d_M(P_{k_2})=R$, call $x_R\in\ptl C$ such that $d(p,x_R)=R$, which can be assumed to be different from an endpoint of $P_2$
(otherwise, $R=r$ and $C$ is then a circle). 
Let $\alpha_1$ be the angle at $p$ determined by the line segments $\overline{p\,v_1}$ and $\overline{p\,x_R}$,
and $\alpha_2$ the angle determined by $\overline{p\,x_R}$ and $\overline{p\,v_2}$.
Clearly $\alpha_1+\alpha_2=\pi$, 
and so we can assume that $\alpha_1\geq\pi/2$. Then
\[
d(v_1,x_R)^2=d(p,x_R)^2+d(p,v_1)^2-2\,d(p,x_R)\,d(p,v_1)\cos(\alpha_1)>d(p,x_R)^2,
\]
yielding
\[
d_M(P_2)\geq d(v_1,x_R)>d(p,x_R)=d_M(P_{k_2}),
\]
as desired.

On the other hand, if $d_M(P_{k_2})=2\,r\sin(\pi/k_2)$, since $k_2>2$, we have that
\[
d_M(P_{k_2})=2\,r\sin(\pi/k_2)<2\,r\sin(\pi/2)=2\,r=d(v_1,v_2)\leq d_M(P_2),
\]
which proves the statement.
\end{proof}

\begin{remark}
For any multi-rotationally symmetric planar convex body $C$ (different from a circle) 
with minimal degree equal to $2$, the corresponding values 
$$\{d_M(P_{k_1}),\ldots, d_M(P_{k_n})\},$$
where $\{k_1=2,\ldots,k_n\}$ is the set of divisors (greater than $1$) 
of the associated maximal degree of $C$, do not \emph{all} coincide, in view of Lemma~\ref{le:2}. 
\end{remark}

\begin{remark}
\label{re:circle2}
If we consider a circle $\mathscr{C}$, then we have that $\mathscr{C}$ is $k$-rotationally symmetric for any $k\in\nn$, $k\geq2$. 
Moreover, the minimal degree $\chi_\mathscr{C}$ is equal to $2$
and, as commented previously in Remark~\ref{re:circle1}, the associated maximal degree would be $+\infty$. 
Straightforward computations using Lemma~\ref{le:dM} give that 
$$d_M(P_2)>d_M(P_3)>d_M(P_4)>d_M(P_5)>d_M(P_6)\geq d_M(P_k),$$
for any $k\in\nn$, $k\geq7$. 
\end{remark}



\begin{remark}
For a multi-rotationally symmetric planar convex body $C$ with minimal degree equal to $2$,
we cannot discuss either which is the \emph{global minimum value} for the maximum relative diameter in this setting,
as in Subsection \ref{subsec:k=3},
since we cannot estimate $$\min\{d_M(P,C):P \text{ 2-partition of }C\}.$$
Recall that such a minimum is not necessarily provided by the standard $2$-partition, as the example from Figure~\ref{fig:dos} shows.
\end{remark}

\noindent {\bf Acknowledgements.} This work was partially supported by the research project MTM2013-48371-C2-1-P,
and by Junta de Andaluc\'ia grant FQM-325.
The author thanks Jes\'us Yepes for suggesting this problem during the RSME Biennial Congress held in Granada in 2015.

\end{document}